\documentclass{amsart}

\usepackage{fullpage}

\usepackage{amsmath}
\usepackage{amsfonts}
\usepackage{amssymb}
\usepackage{graphicx}
\usepackage{mathrsfs}
\usepackage{amsthm}
\usepackage{enumerate}
\usepackage{leftidx}
\usepackage{hyperref}
\usepackage{extarrows}
\usepackage[all]{xy}
\usepackage{stmaryrd}
\usepackage{wasysym}
\usepackage[OT2,T1]{fontenc}
\usepackage{epstopdf}
\usepackage {hyperref}
\numberwithin{equation}{subsubsection}
\newtheorem{thm}[subsubsection]{Theorem}
\newtheorem*{thm*}{Theorem}
\newtheorem*{thmA}{Theorem A}
\newtheorem*{thmB}{Theorem B}

\newtheorem{cor}[subsubsection]{Corollary}
\newtheorem{lem}[subsubsection]{Lemma}
\newtheorem{prop}[subsubsection]{Proposition}
\theoremstyle{definition}
\newtheorem{defn}[subsubsection]{Definition}
\theoremstyle{remark}
\newtheorem{rem}[subsubsection]{Remark}

\DeclareMathOperator{\End}{End}

\DeclareMathOperator{\Hom}{Hom}

\DeclareMathOperator{\Fil}{Fil}

\DeclareMathOperator{\Lie}{Lie}

\DeclareMathOperator{\Gr}{Gr}

\DeclareMathOperator{\lin}{lin}

\DeclareMathOperator{\id}{id}

\DeclareMathOperator{\GL}{GL}

\DeclareMathOperator{\UU}{U}

\newcommand{\RZ}{\mathcal{N}}

\DeclareMathOperator{\Spf}{Spf}

\DeclareMathOperator{\inv}{inv}

\newcommand{\lprod}[1]{\langle#1\rangle}
\newcommand{\adele}{\mathbb A}

\newcommand{\set}[1]{\left\{#1\right\}}

\newcommand{\To}{\longrightarrow}
\newcommand{\isom}{\overset{\sim}{\To}}

\newcommand{\ZZ}{\mathbb{Z}}
\newcommand{\QQ}{\mathbb{Q}}

\newcommand{\FF}{\mathbb{F}}

\newcommand{\oo}{\mathcal{O}}

\newcommand{\LL}{\mathcal L}

\newcommand{\ignore}[1]{}
\DeclareSymbolFont{cyrletters}{OT2}{wncyr}{m}{n}
\DeclareMathSymbol{\Sha}{\mathalpha}{cyrletters}{"58}

\usepackage{enumitem}

\setlist[enumerate]{leftmargin=*}
\setlist[itemize]{leftmargin=*}

\title{Remarks on the arithmetic fundamental lemma}

\author[Chao Li]{Chao Li}\email{chaoli@math.columbia.edu} 
\address{Department of Mathematics, Columbia University, 2990 Broadway,
  New York, NY 10027}

\author[Yihang Zhu]{Yihang Zhu}\email{yihang@math.columbia.edu}
\address{Department of Mathematics, Columbia University, 2990 Broadway,
  New York, NY 10027}


\date{\today}

\begin{document}

\begin{abstract}
  W. Zhang's arithmetic fundamental lemma (AFL) is a conjectural identity between the derivative of an orbital integral on a symmetric space with an arithmetic intersection number on a unitary Rapoport--Zink space. In the minuscule case, Rapoport--Terstiege--Zhang have verified the AFL conjecture via explicit evaluation of both sides of the identity. We present a simpler way for evaluating the arithmetic intersection number, thereby providing a new proof of the AFL conjecture in the minuscule case.
\end{abstract}

\maketitle




\section{Introduction}

\subsection{Zhang's arithmetic fundamental lemma} The \emph{arithmetic Gan--Gross--Prasad conjectures} (arithmetic GGP) generalize the celebrated Gross--Zagier formula to higher dimensional Shimura varieties (\cite[\S 27]{Gan2012}, \cite[\S 3.2]{Zhang2012}). The arithmetic fundamental lemma (AFL) conjecture arises from Zhang's relative trace formula approach for establishing the arithmetic GGP for the group $U(1, n-2)\times U(1, n-1)$. It relates a derivative of orbital integrals on symmetric spaces to an arithmetic intersection number of cycles on unitary Rapoport--Zink spaces,
\begin{equation}
  \label{eq:AFL}
  O'(\gamma, \mathbf{1}_{S_n(\mathbb{Z}_p)})=-\omega(\gamma)\langle \Delta(\mathcal{N}_{n-1}), (\id \times g)\Delta(\mathcal{N}_{n-1})\rangle.
\end{equation}
For the precise definitions of quantities appearing in the identity, see \cite[Conjecture 1.2]{RTZ}. The left-hand side of (\ref{eq:AFL}) is known as the \emph{analytic side} and the right-hand side is known as the \emph{arithmetic-geometric} side. The AFL conjecture has been verified for $n=2,3$ (\cite{Zhang2012}), and for general $n$ in the minuscule case (in the sense that $g$ satisfies a certain minuscule condition) by Rapoport--Terstiege--Zhang \cite{RTZ}. In all these cases, the identity (\ref{eq:AFL}) is proved via explicit evaluation of both sides. When $g$ satisfies a certain inductive condition, Mihatsch \cite{Mihatsch2016} has recently developed a recursive algorithm which reduces the identity (\ref{eq:AFL}) to smaller $n$, thus establishing some new cases of the AFL conjecture.

In the minuscule case, the evaluation of the analytic side is relatively straightforward. The bulk of \cite{RTZ} is devoted to a highly nontrivial evaluation of the arithmetic-geometric side, which is truly a tour de force. Our main goal in this short note is to present a new (and arguably simpler) way to evaluate the arithmetic-geometric side in \cite{RTZ}.

\subsection{The main results} Let $p$ be an odd prime. Let $F=\mathbb{Q}_p$, $k=\overline{\mathbb{F}}_p$, $W=W(k)$ and $K=W[1/p]$. Let $\sigma$ be the $p$-Frobenius acting on $\bar \FF_p$, $W$ and $K$. Let $E=\mathbb{Q}_{p^2}$ be the unramified quadratic extension of $F$.  The \emph{unitary Rapoport--Zink space} $\mathcal{N}_n$ is the formal scheme over $\Spf W$ parameterizing deformations up to quasi-isogeny of height 0 of unitary $p$-divisible groups of signature $(1,n-1)$ (definitions recalled in \S \ref{sec:unit-rapop-zink}). Fix $n\ge2$ and write $\mathcal{N}=\mathcal{N}_n$, $\mathcal{M}=\mathcal{N}_{n-1}$ for short. There is a natural closed immersion $\delta: \mathcal{M}\rightarrow \mathcal{N}$. Denote by $\Delta\subseteq \mathcal{M}\times_W\mathcal{N}$ the image of $(\id,\delta): \mathcal{M}\rightarrow \mathcal{M}\times_W\mathcal{N}$, known as the (local) \emph{diagonal cycle} or \emph{GGP cycle} on $\mathcal{M}\times_W \mathcal{N}$.

Let $C_{n-1}$ be a non-split $\sigma$-Hermitian $E$-space of dimension $n-1$. Let $C_{n}=C_{n-1} \oplus E u$ (where the direct sum is orthogonal and $u$ has norm 1) be a non-split $\sigma$-Hermitian $E$-space of dimension $n$. The unitary group $J=\UU(C_n)$ acts on $\mathcal{N}$ in a natural way (see \S \ref{sec:group-j}). Let $g\in J(\mathbb{Q}_p)$. The arithmetic-geometric side of the AFL conjecture (\ref{eq:AFL}) concerns the arithmetic intersection number of the diagonal cycle $\Delta$ and its translate by $\id\times g$, defined as $$\langle \Delta, (\id\times g)\Delta\rangle:=\log p\cdot \chi(\mathcal{M}\times_W\mathcal{N}, \mathcal{O}_\Delta \otimes^\mathbb{L}\mathcal{O}_{(\id \times g)\Delta}).$$ When $\Delta$ and $(\id\times g)\Delta$ intersect properly, namely when the formal scheme
\begin{equation}
  \label{eq:intersection}
  \Delta \cap(\id\times g)\Delta \cong \delta(\mathcal{M})\cap \mathcal{N}^g
\end{equation}
 is an Artinian scheme (where $\mathcal{N}^g$ denotes the fixed points of $g$), the intersection number is simply $\log p$ times the $W$-length of the Artinian scheme (\ref{eq:intersection}).

 Recall that $g\in J(\mathbb{Q}_p)$ is called \emph{regular semisimple} if $$L(g):=\mathcal{O}_E\cdot u+\mathcal{O}_E\cdot g u+\cdots+ \mathcal{O}_E\cdot g^{n-1} u$$ is an $\mathcal{O}_E$-lattice in $C_n$. In this case, the \emph{invariant} of $g$ is the unique sequence of integers $$\inv(g):=(r_1\ge r_2\ge \ldots\ge r_n)$$ characterized by the condition that there exists a basis $\{e_i\}$ of the lattice $L(g)$ such that $\{p^{-r_i}e_i\}$ is a basis of the dual lattice $L(g)^\vee$. It turns out that the ``bigger'' $\inv(g)$ is, the more difficult it is to compute the intersection. With this in mind, recall that a regular semisimple element $g$ is called \emph{minuscule} if $r_1=1$ and $r_n\ge0$ (equivalently, $pL(g)^\vee\subseteq L(g)\subseteq L(g)^\vee$). In this minuscule case, the intersection turns out to be proper, and one of the main results of \cite{RTZ} is an explicit formula for the $W$-length of (\ref{eq:intersection}) at each of its $k$-point.

To state the formula, assume $g$ is regular semisimple and minuscule, and suppose $\RZ^g$ is nonempty. Then $g$ stabilizes both $L(g)^\vee$ and $L(g)$ and thus acts on the $\mathbb{F}_{p^2}$-vector space $L(g)^\vee/L(g)$. Let $P(T)\in \mathbb{F}_{p^2}[T]$ be the characteristic polynomial of $g$ acting on $L(g)^\vee/L(g)$. For any irreducible polynomial $R(T)\in \mathbb{F}_{p^2}[T]$, we denote its multiplicity in $P(T)$ by $m(R(T))$ and define its \emph{reciprocal} by $$R^*(T):=T^{\deg R(T)}\cdot\sigma( R(1/T)).$$ We say $R(T)$ is \emph{self-reciprocal} if $R(T)=R^*(T)$. By \cite[8.1]{RTZ}, if $(\delta(M)\cap \mathcal{N}^g)(k)$ is nonempty, then $P(T)$ has a unique self-reciprocal monic irreducible factor $Q(T)|P(T)$ such that $m(Q(T))$ is odd. We denote $c:=\frac{m(Q(T))+1}{2}$. Then $1\le c\le \frac{n+1}{2}$. Now we are ready to state the intersection length formula.

 \begin{thmA}[{\cite[Theorem 9.5]{RTZ}}]
Assume $g$ is regular semisimple and minuscule. Assume $p>c$. Then for any $x\in (\delta(M)\cap \mathcal{N}^g)(k)$, the complete local ring of $\delta(M)\cap \mathcal{N}^g$ at $x$ is isomorphic to $k[X]/X^c$, and hence has $W$-length equal to $c$.
\end{thmA}

We will present a simpler proof of Theorem A in Theorem \ref{thm:finalA}. Along the way, we will also give a simpler proof of the following Theorem B in Corollary \ref{cor:reducedness}, which concerns minuscule special cycles (recalled in \S \ref{subsec:special cycles}) on unitary Rapoport--Zink spaces and may be of independent interest.

\begin{thmB}[{\cite[Theorems 9.4, 10.1]{RTZ}}] Let $\mathbf{v}=(v_1,\ldots,v_n)$ be an $n$-tuple of vectors in $C_n$. Assume it is minuscule in the sense that $L(\mathbf v):= \mathrm{span}_{\oo_E} \mathbf{v}$ is an $\oo_E$-lattice in $C_n$ satisfying $ p L(\mathbf v) ^{\vee} \subseteq L(\mathbf v) \subseteq L(\mathbf v) ^{\vee}$. Let $\mathcal{Z}(\mathbf{v})\subseteq \mathcal{N}$ be the associated special cycle.  Then $\mathcal{Z}(\mathbf{v})$ is a reduced $k$-scheme.
\end{thmB}

\subsection{Novelty of the proof} 
The original proofs of Theorems A and B form the technical heart of \cite{RTZ} and occupies its two sections \S 10-\S 11. As explained below, our new proofs presented here have the merit of being much shorter and more conceptual.

\subsubsection{Theorem A}

The original proof of Theorem A uses Zink's theory of windows to compute the local equations of (\ref{eq:intersection}). It requires explicitly writing down the window of the universal deformation of $p$-divisible groups and solving quite involved linear algebra problems. Theorem B ensures that the intersection is entirely concentrated in the special fiber so that each local ring has the form $k[X]/X^\ell$. The assumption $p>c$ ensures $\ell< p$ so that the ideal of local equations is \emph{admissible} (see the last paragraph of \cite[p. 1661]{RTZ}), which is crucial in order to construct the frames for the relevant windows needed in Zink's theory.

Our new proof of Theorem A does not use Zink's theory and involves little explicit computation. Our key observation is that Theorem B indeed allows us to identify the intersection (\ref{eq:intersection}) as the fixed point scheme $\mathcal{V}(\Lambda)^{\bar g}$ of a finite order automorphism $\bar g$ on a generalized Deligne--Lusztig variety $\mathcal{V}(\Lambda)$ (\S \ref{sec:arithm-inters-as}), which becomes purely an algebraic geometry problem over the \emph{residue field} $k$. When $p> c$, it further simplifies to a more elementary problem of determining the fixed point scheme of a finite order automorphism $\bar g\in \GL_{d+1}(k)$ on a projective space $\mathbb{P}^d$ over $k$ (\S \ref{sec:study-mathc-vlambda}). This elementary problem has an answer in terms of the sizes of the Jordan blocks of $\bar g$ (Lemma \ref{lem:jordanblock}), which explains conceptually why the intersection multiplicity should be equal to $c$. Notice that our method completely avoids computation within Zink's theory, and it would be interesting to explore the possibility to remove the assumption $p> c$ using this method.  

\subsubsection{Theorem B}

The original proof of Theorem B relies on showing two things (by \cite[Lemma 10.2]{RTZ}): (1) the minuscule special cycle $\mathcal{Z}(\mathbf{v})$ has no $W/p^2$-points and (2) its special fiber $\mathcal{Z}(\mathbf{v})_k$ is regular. Step (1) is relatively easy using Grothendieck--Messing theory. Step (2) is more difficult: for super-general points $x$ on $\mathcal{Z}(\mathbf{v})_k$, the regularity is shown by explicitly computing the local equation  of $\mathcal{Z}(\mathbf{v})_k$ at $x$ using Zink's theory; for non-super-general points, the regularity is shown using induction and reduces to the regularity of certain special divisors, whose the local equations can again be explicitly computed using Zink's theory.

Our new proof of Theorem B does not use Zink's theory either and involves little explicit computation. Our key observation is that to show both (1) and (2), it suffices to consider the thickenings $\mathcal{O}$ of $k$ which are objects of the crystalline site of $k$. These $\mathcal{O}$-points of $\mathcal{Z}(\mathbf{v})$ then can be understood using only Grothendieck--Messing theory (Theorem \ref{thm:analogue of MP}). We prove a slight generalization of (1) which applies to possibly non-minuscule special cycles (Corollary \ref{cor:no mod p^2 points}). We then prove the tangent space of the minuscule special cycle $\mathcal{Z}(\mathbf{v})_k$ has the expected dimension (Corollary \ref{cor:dim of tangent space}). The desired regularity (2) follows immediately.

\subsubsection{}

 Our new proofs are largely inspired by our previous work on arithmetic intersections on GSpin Rapoport--Zink spaces \cite{rzo}.  
 The GSpin Rapoport--Zink spaces considered in \cite{rzo} are not of PEL type, which makes them technically more complicated. So the unitary case treated here can serve as a guide to \cite{rzo}. We have tried to indicate similarities between certain statements and proofs,  for both clarity and the convenience of the readers.

\subsection{Structure of the paper}

In \S \ref{sec:unit-rapop-zink-1}, we recall necessary backgrounds on unitary Rapoport--Zink spaces and the formulation of the arithmetic intersection problem. In \S \ref{sec:reduc-minusc-spec}, we study the local structure of the minuscule special cycles and prove Theorem B. In \S \ref{sec:inters-length-form}, we provide an alternative moduli interpretation of the generalized Deligne--Lusztig variety $\mathcal{V}(\Lambda)$ and prove Theorem A.

\subsection{Acknowledgments}  We are very grateful to M. Rapoport and W. Zhang for helpful conversations and comments. Our debt to the paper \cite{RTZ} should be clear to the readers. 

\section{Unitary Rapoport--Zink spaces}\label{sec:unit-rapop-zink-1}

In this section we review the structure of unitary Rapoport-Zink spaces. We refer to \cite{Vollaard2010}, \cite{Vollaard2011} and \cite{Kudla2011}  for the proofs of these facts.

\subsection{Unitary Rapoport--Zink spaces}\label{sec:unit-rapop-zink}

Let $p$ be an odd prime. Let $F=\mathbb{Q}_p$, $k=\overline{\mathbb{F}}_p$, $W=W(k)$ and $K=W[1/p]$. Let $\sigma$ be the $p$-Frobenius acting on $\bar \FF_p$, and we also denote by $\sigma $ the canonical lift of the $p$-Frobenius to $W$ and $K$. For any $\FF_p$-algebra $R$, we also denote by $\sigma$ the Frobenius $x\mapsto x^p $ on $R$. \ignore{More generally, for any $\FF_p$-vector space $L$, we denote $$\sigma : L \otimes R \to L\otimes R, ~ l\otimes x\mapsto l \otimes x^p. $$}

  Let $E=\mathbb{Q}_{p^2}$ be the unramified quadratic extension of $F$. Fix a $\QQ_p$-algebra embedding $\phi_0: \mathcal{O}_E\hookrightarrow W$ and denote by $\phi_1$ the embedding $\sigma \circ \phi_0: \oo_E \hookrightarrow W$. The embedding $\phi_0$ induces an embedding between the residue fields $\FF_{p^2} \hookrightarrow k$, which we shall think of as the natural embedding. For any $\oo_E$-module $\Lambda$ we shall write $\Lambda_W$ for $\Lambda \otimes _{\oo_E, \phi_0}  W$.

Let $r,s $ be positive integers and let $n= r+s$. We denote by $\mathcal{N}_{r,s}$ the unitary Rapoport--Zink spaces of signature $(r,s)$, a formally smooth formal $W$-scheme, parameterizing deformations up to quasi-isogeny of height 0 of unitary $p$-divisible groups of signature $(r,s)$. More precisely, for a $W$-scheme $S$, a \emph{unitary $p$-divisible groups} of signature $(r,s)$ over $S$ is a triple $(X, \iota, \lambda)$, where
\begin{enumerate}
\item  $X$ is a $p$-divisible group of dimension $n$ and height $2n$ over $S$,
\item  $\iota: \mathcal{O}_E\rightarrow \End(X)$ is an action satisfying the signature $(r,s)$ condition, i.e., for $\alpha\in\mathcal{O}_E$, $$\mathrm{char}(\iota(\alpha): \Lie X)(T)=(T-\phi_0(\alpha))^r(T-\phi_1(\sigma))^{s}\in \mathcal{O}_S[T],$$
\item $\lambda: X\rightarrow X^\vee$ is a principal polarization such that the associated Rosati involution induces $\alpha\mapsto \sigma(\alpha)$ on $\mathcal{O}_E$ via $\iota$.
\end{enumerate}
Over $k$, there is a unique such triple $(X,\iota, \lambda)$ such that $X$ is supersingular, up to $\mathcal{O}_E$-linear isogeny preserving the polarization up to scalar. Fix such a framing triple and denote it by $(\mathbb{X},\iota_{\mathbb X}, \lambda_{\mathbb X})$.

Let $\mathrm{Nilp}_W$ be the category of $W$-schemes on which $p$ is locally nilpotent. Then the \emph{unitary Rapoport--Zink space} $\mathcal{N}_{r,s}$ represents the functor $\mathrm{Nilp}_W\rightarrow\mathbf{Sets}$ which sends $S\in \mathrm{Nilp}_W$ to the set of isomorphism classes of quadruples $(X, \iota, \lambda,\rho)$, where $(X, \iota, \lambda)$ is a unitary $p$-divisible group  over $S$ of signature $(r,s)$ and $\rho: X\times_S S_k\rightarrow \mathbb{X}\times_k S_k$ is an $\mathcal{O}_E$-linear quasi-isogeny  of height zero which respects $\lambda$ and $\lambda_\mathbb{X}$ up to a scalar $c(\rho)\in \mathcal{O}_F^\times = \ZZ_p ^{\times}$ (i.e., $\rho^\vee\circ \lambda_{\mathbb{X}}\circ\rho=c(\rho)\cdot \lambda$).

In the following we denote $\mathcal{N}:=\mathcal{N}_{1,n-1}$, $\mathcal{M}:=\mathcal{N}_{1,n-2}$ and $\bar{\mathcal{N}_0}:=\mathcal{N}_{0,1}\cong \Spf W$. They have relative dimension $n-1$, $n$ and $0$ over $\Spf W$ respectively. We denote by  $\overline{\mathbb Y} = (\overline{\mathbb{Y}},\iota_{\overline{\mathbb{Y}}}, \lambda_{\overline{\mathbb{Y}}})$ the framing object for $\bar{\mathcal{N}_0}$ and denote by $\bar Y = (\bar Y, \iota_{\bar Y}, \lambda_{\bar Y})$ the universal $p$-divisible group over $\bar{\mathcal{N}_0}$. 
We may and shall choose framing objects $\mathbb X = (\mathbb X, \iota_{\mathbb X}, \lambda_{\mathbb X})$ and $\mathbb X^{\flat} =  (\mathbb X^{\flat} , \iota_{\mathbb X^{\flat} } , \lambda_{\mathbb X^{\flat}})$ for $\mathcal{N}$ and $\mathcal{M}$ respectively such that $$\mathbb X = \mathbb X^{\flat} \times \overline{\mathbb Y } $$  as unitary $p$-divisible groups.

\subsection{The group $J$} \label{sec:group-j}

The covariant Dieudonn\'e module $M=\mathbb{D}(\mathbb{X})$ of the framing unitary $p$-divisible group is a free $W$-module of rank $2n$ together with an $\mathcal{O}_E$-action (induced by $\iota$) and a perfect symplectic $W$-bilinear form $\langle  \cdot , \cdot  \rangle: M\times M\rightarrow W$ (induced by $\lambda$), cf. \cite[\S 2.3]{Vollaard2011}. Let $N=M \otimes_W K$ be the associated isocrystal and extend $\langle \cdot, \cdot \rangle$ to $N$ bilinearly. Let $F, V$ be the usual operators on $N$. We have \begin{align}\label{eq:comp of symp}
\lprod{F x, Fy} =  p \sigma ( \lprod{x,y}), \quad \forall x,y \in N.
\end{align}

The $E$-action decomposes $N$ into a direct sum of two $K$-vector spaces of dimension $n$, 
\begin{align}\label{eq:decomposing N}
N=N_0 \oplus N_1,
\end{align}
where the action of $E$ on $N_i$ is induced by the embedding $\phi_i$. Both $N_0$ and $N_1$ are totally isotropic under the symplectic form. The operator $F$ is of degree one and induces a $\sigma$-linear bijection $N_0\isom N_1$. Since the isocrystal $N$ is supersingular, the degree $0$ and $\sigma^2$-linear operator $$\Phi=V^{-1}F=p^{-1}F^2$$ has all slopes zero (\cite[\S 2.1]{Kudla2011}). We have a $K$-vector space $N_0$ together with a $\sigma^2$-linear automorphism $\Phi$.\footnote{Such a pair $(N_0, \Phi)$ is sometimes called a relative isocrystal.} The fixed points $$C=N_0^\Phi$$ is an $E$-vector space of dimension $n$ and $N_0=C \otimes _{E, \phi_0} K$. Fix $\delta\in \mathcal{O}_E^\times$ such that $\sigma(\delta)=-\delta$. Define a non-degenerate $\sigma$-sesquilinear form on $N_0$ by  
\begin{align}\label{eq:pairing on C}
\{x,y\}:=(p\delta)^{-1}\langle x, Fy\rangle.
\end{align} 
Using (\ref{eq:comp of symp}) it is easy to see that 
\begin{align}\label{eq:generalized Hermitian}
\sigma (\set{x,y}) = \set{\Phi y, x}, \quad \forall x,y \in N_0.
\end{align}

In particular, when restricted to $C$, the form $\set{\cdot, \cdot}$ is $\sigma$-Hermitian, namely 
\begin{align}\label{eq:being Herm}
\sigma(\set{x,y} ) = \set{y,x}, \quad \forall x, y \in C. 
\end{align} 
In fact, $(C,\set{\cdot,\cdot})$ is the unique (up to isomorphism) non-degenerate non-split $\sigma$-Hermitian $E$-space of dimension $n$. Let $J=\UU(C)$ be the unitary group of $(C,\set{\cdot,\cdot})$. It is an algebraic group over $F=\mathbb{Q}_p$. By Dieudonn\'e theory, the group $J(\mathbb{Q}_p)$ can be identified with the automorphism group of the framing unitary $p$-divisible group $(\mathbb{X},\iota_{\mathbb X},\lambda_{\mathbb X})$ and hence acts on the Rapoport--Zink space $\mathcal{N}$.


\subsection{Special homomorphisms}

By definition, the \emph{space of special homomorphisms} is the $\mathcal{O}_E$-module $\Hom_{\mathcal{O}_E}(\overline{\mathbb{Y}}, \mathbb{X})$. There is a natural $\mathcal{O}_E$-valued $\sigma$-Hermitian form on $\Hom_{\mathcal{O}_E}(\overline{\mathbb{Y}}, \mathbb{X})$ given by $$(x,y)\mapsto \lambda_{\overline{\mathbb{Y}} }^{-1}\circ \hat y\circ \lambda_\mathbb{X}\circ x\in \End_{\mathcal{O}_E}(\overline{\mathbb{Y}})\isom \mathcal{O}_E.$$ By \cite[Lemma 3.9]{Kudla2011}, there is an isomorphism of $\sigma$-Hermitian $E$-spaces 
\begin{align}\label{eq:identifying C}
\Hom_{\mathcal{O}_E}(\overline{\mathbb{Y}}, \mathbb{X}) \otimes_{\mathcal{O}_E}E \isom C.
\end{align}
Therefore we may view elements of $C$ as special quasi-homomorphisms.

\subsection{Vertex lattices}
\label{subsec:vertex lattices}
For any $\oo_E$-lattice $\Lambda\subset C$, we define the dual lattice $\Lambda^\vee:=\{x\in C: \set{x, \Lambda}\subseteq \mathcal{O}_E\}$. It follows from the $\sigma$-Hermitian property (\ref{eq:being Herm}) that we have $(\Lambda^{\vee})^{\vee}  = \Lambda$.

A \emph{vertex lattice} is an $\mathcal{O}_E$-lattice $\Lambda\subseteq C$ such that $p \Lambda\subseteq \Lambda^\vee\subseteq \Lambda$. Such lattices correspond to the vertices of the Bruhat--Tits building of the unitary group $\UU(C)$. Fix a vertex lattice $\Lambda$. The \emph{type} of $\Lambda$ is defined to be $t_\Lambda:=\dim_{\mathbb{F}_{p^2}} \Lambda/\Lambda^\vee$, which is always an \emph{odd} integer such that $1\le t_\Lambda\le n$ (cf. \cite[Remark 2.3]{Vollaard2010}). 

We define $\Omega_0 (\Lambda): =\Lambda/\Lambda^\vee$ and equip it with the perfect $\sigma$-Hermitian form $$(\cdot , \cdot): \Omega_0 (\Lambda)\times \Omega_0(\Lambda)\rightarrow \mathbb{F}_{p^2},\quad(x,y) : =  p \set{ \tilde x, \tilde y }\bmod p,$$ where $\set{\cdot, \cdot}$ is the Hermitian form on $C$ defined in (\ref{eq:pairing on C}), and $\tilde x, \tilde y \in \Lambda$ are lifts of $x,y$. 

We define $$\Omega(\Lambda) : = \Omega_0(\Lambda) \otimes _{\FF_{p^2}} k. $$
\begin{rem}\label{rem:to compare}
	To compare with the definitions in \cite{Vollaard2010}, our $\Omega_0(\Lambda)$ is the space $V$ in \cite[(2.11)]{Vollaard2010}, and our pairing $(\cdot, \cdot )$ differs from the pairing $(\cdot, \cdot)$ defined loc. cit. by a factor of the reduction $\bar \delta \in \FF_{p^2} ^{\times}$ of $\delta$.  
\end{rem}
\ignore{
Let $\Omega=\Omega_0 \otimes_{\mathbb{F}_{p^2}}k$ equipped with the induced $k$-valued pairing from extension of scalars: $$\Omega\times \Omega\rightarrow k,\quad (a \otimes  x, b \otimes y) : = a\sigma(b)\cdot (x,y),\quad a,b \in k, x,y\in \Omega_0.$$ This pairing is linear in the first variable and $\sigma$-linear in the second variable. In the following, when we want to emphasize the dependence of $\Omega_0$ and $\Omega$ on $\Lambda$, we will write $\Omega_0( \Lambda)$ and $\Omega(\Lambda)$. 
}
\ignore{
Notice for $x,y\in\Omega$, $$\lprod{x,y}=\sigma\lprod{y,\sigma^{-2}(x)}.$$ For any subspace $U\subseteq \Omega$, we define $\Phi(U):=\sigma^2(U)$ and $U^\perp:=\{ x\in \Omega: \lprod{x,U}=0\}$. Then it follows that $$(U^\perp)^\perp=\Phi(U),\quad \Phi(U^\perp)=\Phi(U)^\perp.$$}

 \subsection{The variety $\mathcal{V}(\Lambda)$} \label{subsec:V(Lambda)}
Let $\Lambda$ be a vertex lattice and let $\Omega_0 = \Omega _0(\Lambda)$. Recall from \S \ref{subsec:vertex lattices} that $\Omega_0$ is an $\FF_{p^2}$-vector space whose dimension is equal to the type $t = t_{\Lambda}$ of $\Lambda$, an odd number. Let $d: = (t -1) /2$. 
We define $\mathcal{V}(\Omega_0)$ to be the closed $\FF_{p^2}$-subscheme of the Grassmannian $\Gr_{d+1}(\Omega_0)$ (viewed as a scheme over $\FF_{p^2}$) such that for any $\FF_{p^2}$-algebra $R$, 
\begin{align}\label{eq:defn of V_Lambda}
\mathcal{V}(\Omega_0)(R)=\{U\subseteq \Omega_0\otimes_{\FF_{p^2}} R: ~ U \mbox{ is an $R$-module local direct summand of rank $d+1$ such that} ~ U^\perp\subseteq U\}.
\end{align} Here $U^{\perp}$ is by definition $\set{v\in \Omega_0 \otimes R: (v, u)_{ R} = 0, ~\forall u \in U}$, where $(\cdot, \cdot) _{R }$ is the $R$-sesquilinear form on $\Omega_0 \otimes R$ obtained from $(\cdot, \cdot)$ by extension of scalars (linearly in the first variable and $\sigma$-linearly in the second variable).
Then $\mathcal{V}(\Omega_0)$ is a smooth projective $\FF_{p^2}$-scheme of dimension $d$ by \cite[Proposition 2.13]{Vollaard2010} and Remark \ref{rem:to compare}. In fact, $\mathcal V(\Omega_0)$ can be identified as a (generalized) Deligne--Lusztig variety,  by \cite[\S 4.5]{Vollaard2011} (though we will not use this identification in the following).

	We write $\mathcal V(\Lambda)$ for the base change of $\mathcal V(\Omega_0) $ from $\FF_{p^2}$ to $k$.

\ignore{ Let $\mathcal{L}=\Phi^{-1}(U^\perp)$, then we have $(\mathcal{L}^\perp, \Phi(\mathcal{L}))=(U, U^\perp)$. The map $U\mapsto \mathcal{L}$ then gives an isomorphism $$\mathcal{V}(\Lambda)(k)\isom\{\mathcal{L}\subseteq \Omega \text{ totally isotropic}: \dim \mathcal{L}=d, \Phi(\mathcal{L})\subseteq \mathcal{L}^\perp\}.$$ It follows that
\begin{align*}
  \mathcal{V}(\Lambda)(k)&=\{\mathcal{L}\subseteq \Omega \text{ Lagrangian}: \dim (\mathcal{L}+\Phi(\mathcal{L}))=d+1\}\\
  &\cong\{(\mathcal{L}_{d-1},\mathcal{L}_d): \mathcal{L}_d \subseteq \Omega \text{ Lagrangian}, \mathcal{L}_{d-1}\subseteq \mathcal{L}_d\cap\Phi\mathcal{L}_d,\dim \mathcal{L}_{d-1}=d-1\},
\end{align*}
where the last bijection is given by $\mathcal L \mapsto (\mathcal L \cap \Phi(\mathcal L), \mathcal L)$.

\begin{rem}\label{rem:moduli}
 Note that since $\mathcal V(\Lambda)$ is a closed $k$-subvariety of $\Gr_{d+1} (\Omega)$, it is characterized by the set of its $k$-points as a subset of $\Gr_{d+1} (\Omega)(k)$. From this it is easy to see that the above moduli interpretation of $\mathcal V(\Lambda) (k)$ in terms of pairs $(\LL_{d-1}, \LL_d)$ also extends to a similar moduli interpretation of $\mathcal V(\Lambda) (R)$ for any $k$-algebra $R$, similarly as in \cite[\S 2.8]{rzo}.   
\end{rem}}

\subsection{Structure of the reduced scheme $\mathcal{N}^\mathrm{red}$}\label{subsec:str of red scheme}  For each vertex lattice $\Lambda\subseteq C$, we define $\mathcal{N}_\Lambda\subseteq \mathcal{N}$ to be the locus where $\rho_{X}^{-1}\circ \Lambda^\vee\subseteq\Hom(\bar Y, X)$, i.e., the quasi-homomorphisms $\rho^{-1}\circ v$ lift to actual homomorphisms for any $v\in \Lambda^\vee$. Then $\mathcal{N}_\Lambda$ is a closed formal subscheme by \cite[Proposition 2.9]{RZ96}. By \cite[\S 4]{Vollaard2011} we have an isomorphism of $k$-varieties 
\begin{align}\label{eq:N_Lambda and V_Lambda}
\mathcal{N}_\Lambda^\mathrm{red}\isom\mathcal{V}(\Lambda).
\end{align}

\subsection{Some invariants associated to a $k$-point of $\RZ$}
We follow \cite[\S 2.1]{Kudla2011}. \label{subsec:invts}

Let $x$ be a point in $\RZ(k)$. Then $x$ represents a tuple $(X, \iota , \lambda, \rho)$ over $k$ as recalled in \S \ref{sec:unit-rapop-zink}. Via $\rho$, we view the Dieudonn\'e module of $X$ as a $W$-lattice $M_x$ in $N$, which is stable under the operators $F$ and $V$. The endomorphism structure $\iota$ induces an action of $\oo_E \otimes _{\ZZ_p} W \cong W \oplus W $ on $M_x$, which is equivalent to the structure of a $\ZZ/2\ZZ$-grading on $M_x$ (into $W$-modules). We denote this grading by $$M_x = \mathrm{gr}_0 M_x \oplus \mathrm{gr}_1 M_x.$$ This grading is compatible with (\ref{eq:decomposing N}) in the sense that $$ \mathrm{gr}_i M_x = M \cap N_i, ~ i = 0,1.  $$ Moreover both $\mathrm{gr}_0 M_x$ and $\mathrm{gr}_1 M_x$ are free $W$-modules of rank $n$. 
 
Consider the $k$-vector space $M_{x,k} : = M_x \otimes _W k$. It has an induced $\ZZ/2\ZZ$-grading, as well as a canonical filtration  $\Fil^1 (M_{x,k}) \subset M_{x,k}$. Explicitly, $\Fil^1 (M_{x,k})$ is the image of $V(M_x) \subseteq M_x$ under the reduction map $M_x\to M_{x,k}$. Define $$\Fil^1 (\mathrm{gr}_i M_{x,k}) : =  \Fil^1 (M_{x,k}) \cap  \mathrm{gr}_i M_{x,k}.$$ Then $$ \Fil^1 (M_{x,k}) = \Fil^1 (\mathrm{gr}_0 M_{x,k}) \oplus \Fil^1 (\mathrm{gr}_1 M_{x,k}), $$ and by the signature $(1,n-1)$ condition we know that $\Fil^1 (\mathrm{gr}_0 M_{x,k})$ (resp. $\Fil^1 (\mathrm{gr}_1 M_{x,k})$) is a hyperplane (resp. line) in $\mathrm{gr}_0 M_{x,k}$ (resp. $\mathrm{gr}_1 M_{x,k}$). 

The symplectic form $\lprod{\cdot,\cdot}$ on $N$ takes values in $W$ on $M_x$, and hence induces a symplectic form on $M_{x,k}$ by reduction. The latter restricts to a $k$-bilinear non-degenerate pairing 
$$ \mathrm{gr}_0 M_{x,k} \times \mathrm{ gr} _1 M_{x,k} \to k.$$ Under the above pairing, the spaces $\Fil^1(\mathrm{gr} _0 M_{x,k})$ and $\Fil^1(\mathrm{gr}_1 M_{x,k})$ are annihilators of each other.  Equivalently, $\Fil^1 (M_{x,k})$ is a totally isotropic subspace of $M_{x,k}$. 
\subsection{Description of $k$-points by special lattices}\label{subsec:description}
For a $W$-lattice $A$ in $N_0$, we define its dual lattice to be $A^{\vee} : = \set{x\in N_0: \set{x,A} \subseteq W}$.
If $\Lambda$ is an $\oo_E$-lattice in $C$, then we have $(\Lambda_W) ^{\vee} = (\Lambda^{\vee}) _W$. In the following we denote both of them by $\Lambda_W^{\vee}$.  

\begin{defn}\label{defn:special lattice}
	A \emph{special lattice} is a $W$-lattice $A$ in $N_0$ such that $$ A^{\vee} \subseteq A \subseteq p^{-1} A^{\vee}$$ and such that $A/ A^{\vee}$ is a one-dimensional $k$-vector space. 
\end{defn}
\begin{rem}\label{rem:diff convention}
	The apparent difference between the above definition and the condition in \cite[Proposition 1.10]{Vollaard2010} (for $i=0$) is caused by the fact that we have normalized the pairing $\set{\cdot, \cdot }$ on $N_0$ differently from loc. cit., using an extra factor $(p\delta) ^{-1}$ (cf. (\ref{eq:pairing on C})). Our normalization is the same as that in \cite{RTZ}. 
\end{rem}
Recall the following result from \cite{Vollaard2010}.
\begin{prop}[{\cite[Proposition 1.10]{Vollaard2010}}]\label{prop:special lattices}
	There is a bijection from $\mathcal N(k)$ to the set of special lattices, sending a point $x$ to $\mathrm{gr}_0 M_x$ considered in \S \ref{subsec:invts}.  \qed
\end{prop}

\begin{rem}\label{rem:when lies in BT stratum}
	Let $x\in \mathcal N(k)$ and let $A$ be the special lattice associated to it by Proposition \ref{prop:special lattices}. Let $\Lambda$ be a vertex lattice. Then $x\in \mathcal N_{\Lambda} (k)$ if and only if $A\subseteq \Lambda_W $, if and only if $\Lambda^{\vee}_W \subseteq A^{\vee}$. (See also Remark \ref{rem:when lies in special cycle} below.)
\end{rem}

\subsection{Filtrations} We introduce the following notation:
\begin{defn}\label{defn:Fil A}
	Let $A$ be a special lattice. Write $A_k: = A\otimes_W k$. Let $x\in \RZ(k)$ correspond to $A$ under Proposition \ref{prop:special lattices}. Thus $A_k = \mathrm{gr}_0 M_{x,k}$. Define $\Fil^1 (A_k):= \Fil ^1 (\mathrm{gr}_0 M_{x,k})$ (cf. \S \ref{subsec:invts}). It is a hyperplane in $A_k$. 
\end{defn}
\begin{lem}\label{lem:two defns of Fil^1 A}
	Let $A$ be a special lattice. Then $\Phi^{-1} (A^{\vee})$ is contained inside $A$, and its image in $A_k$ is equal to $\Fil^1 (A_k)$.  
\end{lem}
\begin{proof}
	Let $A$ correspond to $x\in \RZ(k)$ under Proposition \ref{prop:special lattices}. Then $F, V$ both preserve the $W$-lattice $M_x$ in $N$ (cf. \S \ref{subsec:invts}). By definition, $\Fil^1(M_{x,k})$ is the image of $V(M_x) \subseteq M_x$ under the reduction map $M_x\to M_{x,k}$. Since the operator $V$ is of degree $1$ with respect to the $\ZZ/2\ZZ$-grading, we see that $\Fil^1(A_k)$ is the image of $V(\mathrm{gr}_1 M_x) \subseteq A$ under $A\to A_k$.  It suffices to prove that 
	\begin{align}\label{eq:Phi^-1}
	\Phi^{-1} (A^{\vee}) = V(\mathrm{gr}_1 M_x). 
	\end{align}	
	By the proof of \cite[Proposition 1.10]{Vollaard2010}, we have $\mathrm{gr}_1 M_x = F^{-1} A^{\vee}$. (Note that because of the difference of normalizations as discussed in Remark \ref{rem:diff convention}, what is denoted by $A^{\vee}$ here is denoted by $p A^{\vee}$ in \cite{Vollaard2010}. Also note that the integer $i$ appearing loc. cit. is $0$ in our case.) Therefore $V(\mathrm{gr}_1 M_x) = V(F^{-1} A^{\vee}). $ But $V F^{-1} = (V^{-1} F)^{-1} = \Phi^{-1}$ because $VF = FV =p$. Thus (\ref{eq:Phi^-1}) holds as desired.  
\end{proof}

\subsection{Special cycles}\label{subsec:special cycles} Let $\mathbf{v}$ be an arbitrary subset of $C$. We define the \emph{special cycle} $\mathcal{Z}(\mathbf{v})\subseteq \RZ$ to be the locus where $\rho^{-1}\circ v\in\Hom(\bar Y, X)$ for all $v\in \mathbf v$, i.e., all the quasi-homomorphisms $\rho^{-1}\circ v$ lift to actual homomorphisms. Note that  $\mathcal Z(\mathbf v)$ only depends on the $\mathcal{O}_E$-submodule $L(\mathbf{v})$ spanned by $\mathbf{v}$ in $C$, and we have $\mathcal Z(\mathbf v) = \mathcal Z( L(\mathbf v) )$.

We say $\mathbf{v}$ is \emph{minuscule} if $L(\mathbf v)$ is an $\oo_E$-lattice in $C$ satisfying $ p L(\mathbf v) ^{\vee} \subseteq L(\mathbf v) \subseteq L(\mathbf v) ^{\vee}$, or equivalently, if $L(\mathbf{v})$ is the dual of a vertex lattice. When this is the case we have $\mathcal Z(\mathbf v) = \mathcal N_{L(\mathbf{v})^\vee}$ by definition. 

\subsection{The intersection problem} \label{subsec:intersection problem}
 Let $C^\flat$ be the analogue for $\mathcal{M}$ of the Hermitian space $C$. Then $C\cong C^\flat \oplus Eu$ for some vector denoted by $u$ which is of norm $1$ and orthogonal to $C^{\flat}$. We have a closed immersion $$\delta: \mathcal{M}\rightarrow\mathcal{N},$$sending $(X,\iota,\lambda,\rho)$ to $(X\times \bar Y, \iota\times \iota_{\bar Y}, \lambda\times \lambda_{\bar Y},\rho\times \id)$. We have $\delta(\mathcal{M})=\mathcal{Z}(u)$. The closed immersion $\delta$ induces a closed immersion of formal schemes $$(\id, \delta): \mathcal{M}\rightarrow \mathcal{M} \times_W \mathcal{N}.$$ Denote by $\Delta$ the image of $(\id,\delta)$, which we call the (local) \emph{GGP cycle}.  For any $g\in J(\mathbb{Q}_p)$, we obtain a formal subscheme $$(\id\times g)\Delta\subseteq \mathcal{M}\times_W \mathcal{N},$$ via the action of $g$ on $\RZ$. Let $g\in J(\QQ_p)$ and let $\RZ^g\subseteq \RZ$ be the fixed locus of $g$. Then by definition we have $$\Delta\cap (\id\times g)\Delta\cong\delta(\mathcal{M})\cap \RZ^g.$$

Our goal is to compute the arithmetic intersection number $$\langle \Delta, (\id\times g)\Delta\rangle,$$ when $g$ is \emph{regular semisimple} and \emph{minuscule} (as defined in the introduction). Notice that $g\in J(\mathbb{Q}_p)$ is regular semisimple if and only if $\mathbf{v}(g):=(u, gu,\ldots, g^{n-1}u)$ is an $E$-basis of $C$. Also notice that a regular semisimple element $g$ is minuscule if and only if $\mathbf{v}(g)$ is minuscule in the sense of \S \ref{subsec:special cycles}.

\section{Reducedness of minuscule special cycles} \label{sec:reduc-minusc-spec}
\subsection{Local structure of special cycles}

\begin{defn}
Let $\mathscr C$ be the following category: 
		\begin{itemize}
			\item Objects in $\mathscr C$ are triples $(\oo, \oo\to k, \delta)$, where $\oo$ is a local Artinian $W$-algebra, $\oo \to k$ is a $W$-algebra map, and $\delta$ is a nilpotent divided power structure on $\ker (\oo \to k)$ (cf. \cite[Definitions 3.1, 3.27]{berthelotogus}).
			\item Morphisms in $\mathscr C$ are $W$-algebra maps that are compatible with the structure maps to $k$ and the divided power structures.  
		\end{itemize}
\end{defn}
\subsubsection{}\label{a construction}
Let $x\in \RZ(k)$ correspond to a special lattice $A$ under Proposition \ref{prop:special lattices}. Let $\oo\in \mathscr C$. By a \emph{hyperplane} in $A_{\oo}:  = A\otimes_W \oo$ we mean a free direct summand of $A_{\oo}$ of rank $n-1$. We define the $\ZZ/2\ZZ$-grading on $M_{x,\oo } : = M_x\otimes _W \oo$ by linearly extending that on $M_x$ (cf. \S \ref{subsec:invts}). Denote by $\widehat{\RZ} _x$ the completion of $\RZ$ at $x$. For any $\tilde x \in \widehat{\RZ} _x (\oo)$, we have a unitary $p$-divisible group of signature $(1,n-1)$ over $\oo$ deforming that over $k$ defined by $x$.  
By Grothendieck-Messing theory, we obtain the Hodge filtration $\Fil ^1_{\tilde x} M_{x, \oo} \subseteq  M_{x,\oo} $. Define $f_{\oo} (\tilde x)$ to be the intersection 
$$ \Fil ^1 _{\tilde x} M_{x, \oo} \cap \mathrm{gr_0} M_{x , \oo}$$ inside $M_{x ,\oo}$. By the signature $(1,n-1)$ condition, $f_{\oo} (\tilde x)$ is a hyperplane in $A_{\oo}$. It also lifts $\Fil ^1 A_k$ (cf. Definition \ref{defn:Fil A}) by construction. Thus we have defined a map \begin{align}
\label{eq:f_O}
f_{\oo} : \widehat{\mathcal N} _x (\oo ) \isom \set{\mbox{hyperplanes in $A_{\oo}$ lifting }\Fil^1 A_k }.
\end{align} By construction, $f_{\oo}$ is functorial in $\oo$ in the sense that the collection $(f_{\oo}) _{\oo \in \mathscr C}$ is a natural transformation between two set-valued functors on $\mathscr C$. Here we are viewing the right hand side of (\ref{eq:f_O}) as a functor in $\oo$ using the base change maps. 

The following result is the analogue of \cite[Theorem 4.1.7]{rzo}. As a direct consequence of the PEL moduli problem, it should be well known to the experts and is essentially proved in \cite[Proposition 3.5]{Kudla2011}.
\begin{thm}\label{thm:analogue of MP}Keep the notations in \S \ref{a construction}. 
	\begin{enumerate}
		\item The natural transformation $(f_{\oo}) _{\oo\in \mathscr C}$ is an isomorphism.
		\item Let $\mathbf v$ be a subset of $C$. If $x \in \mathcal Z(\mathbf v)(k)$, then $\mathbf v \subseteq A$. Suppose $x\in \mathcal{Z}(\mathbf{v})(k)$. Then for any $\oo \in \mathscr C$ the map $f_{\oo}$ induces a bijection 
		$$ \widehat{\mathcal Z(\mathbf v) } _{x}(\oo) \isom  $$$$ \set{\mbox{hyperplanes in $A_{\oo}$ lifting }\Fil^1 A_k \mbox{ and containing the image of  $\mathbf v$ in $A_{\oo}$ }}.$$ 
	\end{enumerate}

\end{thm}
 
\begin{proof}
	\textbf{(1)} We need to check that for all $\oo\in \mathscr C$ the map $f_{\oo}$ is a bijection.  Let $\tilde x\in \widehat{\RZ}_x(\oo).$ This represents a deformation over $\oo$ of the $p$-divisible group at $x$. Similarly to the situation in \S \ref{subsec:invts}, the compatibility with the endomorphism structure implies that  $$ \Fil ^1_{\tilde x} M_{x,\oo } =  \bigoplus_{i = 0} ^1 \Fil ^1 _{\tilde x} M_{x,\oo} \cap \mathrm{gr}_i M_{x,\oo}. $$ By the compatibility with the polarization, we know that $\Fil^1_{\tilde x} M_{x,\oo}$ is totally isotropic under the symplectic form on $M_{x, \oo}$. It follows that the two modules $\Fil ^1_{\tilde x} M_{x, \oo} \cap \mathrm{gr} _1 M_{x, \oo}$ and $\Fil ^1_{\tilde x} M_{x, \oo} \cap \mathrm{gr} _0 M_{x, \oo}$ are annihilators of each other if we identify $\mathrm{ gr} _1 M_{x, \oo}$ as the $\oo$-linear dual of $\mathrm{gr} _0 M_{x, \oo}$ using the symplectic form on $M_{x, \oo}$. Therefore, $\Fil^1_{\tilde x} M_{x, \oo}$ can be recovered from $f_{\oo} (\tilde x)$. This together with Grothendieck-Messing theory proves the injectivity of $f_{\oo}$. The surjectivity of $f_{\oo}$ also follows from Grothendieck-Messing theory and the above way of reconstructing $\Fil ^1_{\tilde x} M_{x,\oo}$ from its intersection with $\mathrm{gr}_0 M_{x,\oo}$. Note that the unitary $p$-divisible groups reconstructed in this way do satisfy the signature condition because we have started with \emph{hyperplanes} in $A_{\oo}$. 
	
\textbf{(2)} The statements follow from the proof of \cite[Proposition 3.5]{Kudla2011} and the definition of (\ref{eq:identifying C}) in \cite[Lemma 3.9]{Kudla2011}. We briefly recall the arguments here. If $\phi \in \Hom_{\oo_E} (\overline{\mathbb Y} , \mathbb X) \otimes_{\oo_E} E$ is a special quasi-homomorphism, the element $v\in C$ corresponding to $\phi$ under (\ref{eq:identifying C}) is by definition the projection to $N_0$ of $\phi_*(\bar 1_0) \in N$, where $\phi_*$ is the map $\mathbb D( \overline{\mathbb Y}) \otimes _W K \to \mathbb D (\mathbb X ) \otimes _W K = N$ induced by $\phi$, and $\bar 1_0$ is a certain fixed element in $\mathbb D(\overline{\mathbb Y})$. In fact, $\bar 1_0$ is chosen such that 
\begin{itemize}
	\item $W \bar 1_0= \mathrm{gr}_0  \mathbb D (\overline{\mathbb Y})$ , where the grading is with respect to the $\oo_E$-action on $\overline{\mathbb Y}$,
	\item $W \bar 1_0 = \Fil^1 _{\bar Y} \mathbb D (\overline {\mathbb Y})$,  the Hodge filtration for the deformation $\bar Y$ of $\overline{\mathbb Y}$ over $W$.
\end{itemize}
In particular $v$ and $\phi $ are related by the formula $v= \phi_*(\bar 1_0)$, as the projection to $N_0$ is not needed. 

From now on we assume without loss of generality that $\mathbf v = \set{v}$, with $v$ corresponding to $\phi$ as in the above paragraph. If $x\in \mathcal Z(v) (k)$, then $\phi_*$ has to map $\mathbb D(\overline{\mathbb Y})$ into $M_x$, so $v\in M_x$. Since $\phi_*$ is compatible with the $\ZZ/2\ZZ$-gradings, we further have $v\in A$. We have shown that if $x\in \mathcal Z(v) (k)$, then $v\in A$.

Now suppose $x\in \mathcal{Z}(\mathbf{v})(k)$. Let $\oo \in \mathscr C$. Write $v_{\oo} : = v\otimes 1 \in A_{\oo} \subset M_{x,\oo}$.  For all $\tilde x \in \widehat{\RZ}_x (\oo)$, by Grothendieck-Messing theory we know that $\tilde x \in \widehat{\mathcal Z (v) } _x (\oo)$ if and only if the base change of $\phi_*$ to $\oo$ (still denoted by $\phi_*$) preserves the Hodge filtrations, i.e. $$\phi_* (\Fil^1_{\bar Y} \mathbb D(\overline{\mathbb Y})) \subseteq \Fil^1_{\tilde x} M_{x,\oo} .$$ Since $W \bar 1_0 = \Fil^1_{\bar Y} \mathbb D(\overline{\mathbb Y})$, this last condition is equivalent to $v_{\oo} \in \Fil^1_{\tilde x} M_{x,\oo}$. Again, because $\phi_*$ is compatible with the $\ZZ/2\ZZ$-gradings, the last condition is equivalent to $v_{\oo} \in f_{\oo} (\tilde x)$. In conclusion, we have shown that $\tilde x \in \widehat{\RZ}_x (\oo)$ is in $ \widehat{\mathcal Z (v) } _x (\oo)$ if and only if $v_{\oo} \in f_{\oo} (\tilde x)$, as desired.    
\end{proof}

\begin{cor}\label{cor:k point in special cycle}
	Let $x \in \RZ(k)$ correspond to the special lattice $A$. Let $\mathbf v$ be a subset of $C$. Then $x\in \mathcal Z(\mathbf v) (k)$ if and only if $\mathbf v \subseteq A^{\vee}$.
\end{cor}
\begin{proof}
By part (2) of Theorem \ref{thm:analogue of MP} applied to $\oo = k$, we see that $x \in \mathcal Z(\mathbf v)(k)$ if and only if $\mathbf v \subseteq A$ and the image of $\mathbf v$ in $A_k$ is contained in $\Fil^1 (A_k)$. The corollary follows from Lemma \ref{lem:two defns of Fil^1 A} and the $\Phi$-invariance of $\mathbf v$.  
\end{proof}

\begin{rem}\label{rem:when lies in special cycle}
	  Note that Remark \ref{rem:when lies in BT stratum} is a special case of Corollary \ref{cor:k point in special cycle}. 
\end{rem}

\subsection{Proof of the reducedness}

\begin{cor}\label{cor:no mod p^2 points}
	Let $\Lambda$ be an $\oo_E$-lattice in $C$ with $p^i \Lambda \subseteq \Lambda^{\vee} \subseteq \Lambda$ for some $i \in \ZZ_{\geq 1}$. Then the special cycle $\mathcal Z(\Lambda^{\vee})$ defined by $\Lambda ^{\vee}$ has no $ (W/p^{i+1})$-points. In particular, taking $i=1$ we see that $\RZ_{\Lambda} (W/p^2) = \emptyset$ for any vertex lattice $\Lambda$.
\end{cor}
\begin{proof}
	 
	Let $\oo = W/p^{i+1}$, equipped with the reduction map $W/p^{i+1} \to k$ and the natural divided power structure on the kernel $p \oo$. Then $\oo \in \mathscr C$. Assume $\mathcal Z(\Lambda^{\vee})$ has an $\oo$-point $\tilde x$ reducing to a $k$-point $x$. Let $A$ be the special lattice corresponding to $x$ (cf. \S \ref{subsec:description}). By Theorem \ref{thm:analogue of MP}, there exists a hyperplane $P$ in $A_{\oo}$ lifting $\Fil^1 (A_k)$, such that $P \supseteq \Lambda^{\vee} \otimes_{\oo_E} \oo$. Since $P$ is a hyperplane in $A_{\oo}$, there exists an element $l \in \Hom_{\oo} (A_{\oo}, \oo)$ such that 
	\begin{align}\label{eq:inductive condition for l}
	l(P) = 0 \quad \mbox{ and } \quad l(A_{\oo}) = \oo.
	\end{align}
	We may find an element $\tilde l \in A^{\vee} \subseteq  N_0$ to represent $l$, in the sense that for all $a\otimes 1 \in A_{\oo}$ with $a\in A$, we have 
	$$ l (a\otimes 1) = \mbox{image of }\{a, \tilde l\} \mbox{ under } W\to \oo. $$
Since $l(\Lambda^{\vee} \otimes \oo) \subseteq l (P) =0$, we know that $\{v, \tilde l \}\in  p^{i+1} W$ for all $v\in \Lambda^{\vee}$. Since $\Lambda^{\vee} \subseteq C = N_0 ^{\Phi}$, applying (\ref{eq:generalized Hermitian}) we see that $\{ \tilde l, v \} \in p^{i+1} W $ for all $v\in \Lambda^{\vee}$  . Therefore 
 $$ p^{-i-1} \tilde l \in (\Lambda^{\vee} _W) ^{\vee} = \Lambda_W,$$ and thus $\tilde l \in p^{i+1 } \Lambda_W$, which is contained in $  p \Lambda^{\vee} _W$ by hypothesis.  Since $\Lambda$ is $\Phi$-invariant, we also have $\Phi (\tilde l) \in p \Lambda^{\vee} _W$. But $ \Lambda_W^{\vee} \subseteq A^{\vee}$ by Corollary \ref{cor:k point in special cycle}, so $\Phi(\tilde l) \in pA^{\vee}$.
 It follows that for all $a\in A$, we have $\{ \Phi (\tilde l)  ,a \} \in p W$, and therefore $$\{a, \tilde l \} \xlongequal{ (\ref{eq:generalized Hermitian})} \sigma^{-1} (\{\Phi (\tilde l) , a\}) \in p W$$ contradicting with the second condition in (\ref{eq:inductive condition for l}).      
 \end{proof}
\begin{cor}\label{cor:dim of tangent space}
Let $\Lambda$ be a vertex lattice of type $t$ and let $x \in \mathcal N_{\Lambda} (k)$. Then the tangent space $\mathcal T_{x} \mathcal N_{\Lambda, k}$ to $\mathcal N_{\Lambda, k}$ at $x$, where $\mathcal N_{\Lambda, k}$ is the special fiber (i.e. base change to $k$) of $\mathcal N_{\Lambda}$, is of $k$-dimension $(t-1)/2$. 
\end{cor}
\begin{proof}
This can be deduced from Theorem \ref{thm:analogue of MP} elementarily, in the same way as in \cite[\S 4.2]{rzo}. Here we provide a shorter proof. Firstly we make an easy observation. Denote by $\mathscr C_k$ the full subcategory of $\mathscr C$ consisting of characteristic $p$ objects. Let $\mathcal W_1, \mathcal W_2$ be two formal schemes over $k$. Fix $y_i \in \mathcal W_i (k), i= 1,2$. For $i=1,2$, define the set-valued functor $\mathcal F_i$ on $\mathscr C_k $ sending $\oo$ to the set of $\oo$-points of $\mathcal W_i$ which induce $y_i$ under the structure map $\oo \to k$. Assume $\mathcal F_1 \cong \mathcal F_2$. Then the tangent spaces $\mathcal T_{x_i} \mathcal W_i$ are isomorphic. In fact, this observation is a direct consequence of the definition of the vector space structure on the tangent spaces from the point of view of functor of points, as recalled in the proof of \cite[Lemma 4.2.6]{rzo} for instance. 

Denote by $B$ the $k$-subspace of $A_k$ spanned by the image of $\Lambda^{\vee} $ in $A_k$. 
Consider the Grassmannian $\Gr_{n-1} (A_k)$ parametrizing hyperplanes in the $n$-dimensional $k$-vector space $A_k$. Let $\mathcal W_1$ be the  sub-variety of $\Gr_{n-1} (A_k )$ defined by the condition that the hyperplane should contain $B$, and let $y_1\in \mathcal W_1(k)$ corresponding to $\Fil ^1 (A_k )\subseteq A_k$. Let $\mathcal W_2: = \RZ_{\Lambda, k}$ and $y_2: = x$. By Theorem \ref{thm:analogue of MP}, the assumption on $(\mathcal W_i ,y_i), i =1,2$ in the previous paragraph is satisfied. Hence it suffices to compute the dimension of $\mathcal T_{y_1} \mathcal W_1$. Note that $\mathcal W_1$ is itself a Grassmannian, parametrizing hyperplanes in $A_k / B$. The proof is finished once we know that $A_k/B $ has $k$-dimension $(t+1)/2$. But this is true by the ($\sigma$-linear) duality between the $k$-vector spaces $ A_k/ B = A/ \Lambda^{\vee}_W$ and $\Lambda_W/ A^{\vee}$ under the $\sigma$-sesquilinear form on $\Omega(\Lambda)$ obtained by extension of scalars from the $\sigma$-Hermitian form $(\cdot, \cdot )$ on $\Omega_0(\Lambda)$ (cf. \S \ref{subsec:vertex lattices}) and the fact that $A/ A^{\vee}$ is a $1$-dimensional $k$-vector space (cf. Definition \ref{defn:special lattice}). 
\end{proof}
In the following corollary we re-prove \cite[Theorems 9.4, 10.1]{RTZ}. 

\begin{cor}\label{cor:reducedness}
	Let $\Lambda$ be a vertex lattice. Then $\mathcal N_{\Lambda} = \mathcal N_{\Lambda, k} = \mathcal N_{\Lambda} ^{\mathrm{red}}$ and it is regular.
\end{cor}
\begin{proof}Let $t$ be the type of $\Lambda$. Recall from \S \ref{subsec:str of red scheme} that $\mathcal N_{\Lambda} ^{\mathrm{red} }$ is a smooth $k$-scheme of dimension $(t-1)/2$. By Corollary \ref{cor:dim of tangent space}, all the tangent spaces of $\mathcal N_{\Lambda ,k}$ have $k$-dimension $(t-1)/2$, and so $\mathcal N_{\Lambda ,k}$ is regular. In particular $\mathcal N_{\Lambda ,k}$ is reduced, namely  $\mathcal N_{\Lambda ,k} = \mathcal N_{\Lambda} ^{\mathrm{red}}$. Knowing that $\mathcal N_{\Lambda ,k}$ is regular, and that $\RZ_{\Lambda}$ has no ($W/p^2$)-points (Corollary \ref{cor:no mod p^2 points}), it follows that $\mathcal N_{\Lambda } = \mathcal N_{\Lambda ,k}$ by the general criterion \cite[Lemma 10.3]{RTZ}. 
\end{proof}
\section{The intersection length formula}\label{sec:inters-length-form}
\subsection{The arithmetic intersection as a fixed point scheme}\label{sec:arithm-inters-as}
Fix a regular semisimple and minuscule element $g\in J(\QQ_p)$. Let $L: = L(\mathbf{ v} (g))$ and $\Lambda : = L^{\vee}$. They are both $\oo_E$-lattices in $C$. Recall from the end of \S \ref{subsec:special cycles} that $\Lambda$ is a vertex lattice and $\mathcal Z (L) = \mathcal N_ { \Lambda}$. From now on we assume $\mathcal N^g (k)\neq \emptyset$. As shown in \cite[\S 5]{RTZ}, this assumption implies that both $L$ and $\Lambda$ are $g$-cyclic and stable under $g$. In particular, the natural action of $g$ on $\RZ$ stabilizes $\RZ_{\Lambda}$.  

Let $\Omega _0 = \Omega_0 (\Lambda)$ and $\Omega = \Omega (\Lambda)$. Let $t = t_{\Lambda}$ and $d = (t-1)/2$ as in \S \ref{subsec:V(Lambda)}. Let $\bar g\in \GL(\Omega_0)(\FF_{p^2})$ be the induced action of $g$ on $\Omega_0 $. Then $\bar g$ preserves the Hermitian form $(\cdot , \cdot )$ on $\Omega_0$ and hence acts on $\mathcal V(\Lambda)$. It is clear from the definition of the isomorphism (\ref{eq:N_Lambda and V_Lambda}) given in \cite[\S 4]{Vollaard2011} that it is equivariant for the actions of $g$ and $\bar g$ on the two sides.

\begin{rem}\label{rem:unique block}
Since both $\Lambda$ and $\Lambda^{\vee}$ are $g$-cyclic, the linear operator $\bar g \in \GL(\Omega_0) (\FF_{p^2})$ has equal minimal polynomial and characteristic polynomial. Equivalently, in the Jordan normal form of $\bar g$ (over $k$) there is a unique Jordan block associated to any eigenvalue.
\end{rem}

\begin{prop}\label{prop:scheme}
	$\delta (\mathcal M) \cap \mathcal N^g$ is a scheme of characteristic $p$ (i.e. a $k$-scheme) 
        isomorphic to $\mathcal V(\Lambda) ^{\bar g}$. 
\end{prop}
\begin{proof}
	Recall from \S \ref{subsec:intersection problem} that $\delta(\mathcal M) = \mathcal Z(u)$. Since the $\oo_E$-module $L$ is generated by $u, gu, \cdots, g^{n-1} u$ and stable under $g$, we have $\delta (\mathcal M) \cap \RZ^g = \mathcal Z(L)^g = \RZ_{\Lambda}^g$. By Corollary \ref{cor:reducedness}, we know that $\RZ_{\Lambda}^g = (\RZ_{\Lambda} ^{\mathrm{red}})^g$ . But the latter is isomorphic to the characteristic $p$ scheme $\mathcal V(\Lambda) ^{\bar g}$ under (\ref{eq:N_Lambda and V_Lambda}). 
\end{proof}

\subsection{Study of $\mathcal V(\Lambda) ^{\bar g}$}\label{sec:study-mathc-vlambda}
We start with an alternative moduli interpretation of $\mathcal V(\Omega_0)$. The idea is to rewrite (in Lemma \ref{lem:alternative description}) the procedure of taking the complement $U \mapsto U^{\perp}$ with respect to the Hermitian form, in terms of taking the complement with respect to some quadratic form and taking a Frobenius. The alternative moduli interpretation is given in Corollary \ref{cor:alternative description} below. 

Let $\Theta_0$ be a $t$-dimensional non-degenerate quadratic space over $\mathbb{F}_p$. 
Let $[\cdot,\cdot]: \Theta_0\times \Theta_0 \to \FF_p$ be  the associated bilinear form. Since there is a unique isomorphism class of non-degenerate $\sigma$-Hermitian spaces over $\FF_{p^2}$, we may assume that $\Omega_0 = \Theta_0 \otimes _{\FF_p} \FF_{p^2}$ and that the $\sigma$-Hermitian form $(\cdot, \cdot)$ (cf. \S \ref{subsec:V(Lambda)}) is obtained by extension of scalars (linearly in the first variable and $\sigma$-linearly in the second variable) from $[\cdot,\cdot]$.


\begin{defn}\label{def:sigmastar}
Let $R$ be an $\FF_p$-algebra. We define $[\cdot,\cdot]_R$ to be the $R$-bilinear form on $\Theta_0 \otimes_{\FF_p} R$ obtained from $[\cdot, \cdot]$ by extension of scalars. For any $R$-submodule $\LL \subset\Theta_0 \otimes_{\FF_p} R$, define $$\LL^{\lin \perp}: = \set{v\in \Theta_0 \otimes _{\FF_p} R :  [v,l] _R =0 ,~ \forall l\in \mathcal L}. $$ Define $\sigma_* (\LL)$ to be the $R$-module generated by the image of $\LL $ under the map $$ \sigma : \Theta_0 \otimes _{\FF_p} R \to \Theta_0 \otimes _{\FF_p} R,  \quad v\otimes r \mapsto v\otimes r^p. $$   
\end{defn}

Let $R$ be an $\FF_{p^2}$-algebra. Let $U $ be an $R$-submodule of $\Omega_0 \otimes _{\FF_{p^2}} R$. Since $\Omega_0 \otimes _{\FF_{p^2} } R = \Theta_0 \otimes_{\FF_p} R$, we may view $U$ as an $R$-submodule of the latter and define $\sigma_*(U)$ as in Definition~\ref{def:sigmastar}.

\begin{lem}\label{lem:alternative description} We have $
	U^{\perp} =   (\sigma_* (U) ) ^{\lin \perp}.
$
\end{lem}
\begin{proof}
	Consider two arbitrary elements $$x = \sum _j u_j \otimes r_j ,\quad y = \sum _k v_k \otimes s_k$$ of $\Theta_0 \otimes_{\FF_p} R$. We have 
	$$ (y,x) _R =  \sum_{j,k} s_k r_j^p \cdot  [v_k , u_j ] _R =  \sum_{j,k} s_k r_j^p \cdot  [u_j , v_k ] _R  =  [ \sigma(x), y]_R.$$ Hence for $y\in \Theta_0 \otimes_{\FF_p} R$, we have $y \in U^{\perp}$ if and only if $(y,x )_R = 0$ for all $x\in U$, if and only if $[\sigma(x), y]_R =0$ for all $x\in U$, if and only if $y \in (\sigma_*(U) )^{\lin \perp}$.
\end{proof} 
\begin{cor}\label{cor:alternative description}
	For any $\FF_{p^2}$-algebra $R$, the set $\mathcal V(\Omega_0) (R)$ is equal to the set of $R$-submodules $U$ of $$\Omega_0 \otimes_{\FF_{p^2}} R =  \Theta_0\otimes_{\FF_{p}} R, $$ such that $U$ is an  $R$-module local direct summand of rank $d+1$, satisfying$$(\sigma_*(U)) ^{\lin \perp}\subseteq U.$$

\end{cor}
\begin{proof}
	This is a direct consequence of (\ref{eq:defn of V_Lambda}) and Lemma \ref{lem:alternative description}. 
 \end{proof}

 In the following we denote $\mathcal V(\Lambda) $ by $\mathcal V$ for simplicity, where $\Lambda$ is always fixed as in the beginning of \S \ref{sec:arithm-inters-as}. 
 Denote $\Theta:  = \Theta_0 \otimes_{\FF_p} k$. Fix a point $x_0 \in \mathcal V ^{\bar g} (k)$. Let $U_0$ correspond to $x_0$ under (\ref{eq:defn of V_Lambda}) or Corollary \ref{cor:alternative description}. Define $$\LL_{d+1}: = U_0, \quad  \LL_d: = (\sigma_*(U_0)) ^{\lin\perp} \xlongequal
 {\mbox{Lemma \ref{lem:alternative description}}} U_0^{\perp}. $$ They are subspaces of $\Theta$ stable under $\bar g$, of $k$-dimensions $d+1$ and $d$ respectively.

 \begin{defn}\label{eq:defn of I}
   Define $\mathcal I: = \mathbb P (\Theta/\LL_d)$, a projective space of dimension $d$ over $k$.
 \end{defn}
 
  Then $ \LL_{d+1}$ defines an element in $\mathcal I(k)$, which we still denote by $x_0$ by abuse of notation. We have a natural action of $\bar g$ on $\mathcal I$ that fixes $x_0$. Let $\mathcal R_p$ (resp. $\mathcal S_p$) be the quotient of the local ring of $\mathcal I^{\bar g}$ (resp. of $\mathcal V^{\bar g}$) at $x_0$ divided by the $p$-th power of its maximal ideal. 
\begin{lem}\label{lem:local model}
	There is a $k$-algebra isomorphism $\mathcal R_p \cong \mathcal S_p$.
\end{lem} 
\begin{proof}
		The proof is based on exactly the same idea as \cite[Lemma 5.2.9]{rzo}.
	Let $\tilde {\mathcal R}_p$ (resp. $\tilde {\mathcal S}_p$) be the quotient of the local ring of $\mathcal I$ (resp. of $\mathcal V$) at $x_0$ divided by the $p$-th power of its maximal ideal. 
 Let $R$ be an arbitrary local $k$-algebra with residue field $k$ such that the $p$-th power of its maximal ideal is zero. Then by Lemma \ref{lem:alternative description}, the $R$-points of $\mathcal V$ lifting $x_0$ classify $R$-module local direct summands $U$ of $\Theta \otimes _k R$ of rank $d+1$ that lift $\LL_{d+1}$, and such that $$ U\supseteq (\sigma_*(U)) ^{\lin \perp}. $$ But by the assumption that the $p$-th power of the maximal ideal of $R$ is zero, we have $$\sigma_{R,*} (U) = (\sigma_{k,*} (\LL_{d+1})) \otimes _ k R ,$$ where we have written $\sigma_R$ and $\sigma_k$ to distinguish between the Frobenius on $R$ and on $k$. Therefore 
 $$  (\sigma_{R,*}(U)) ^{\lin \perp} = \big ( (\sigma_{k,*} (\LL_{d+1})) \otimes _ k R \big) ^{\lin \perp} =  (\sigma_{k,*} (\LL_{d+1}))^{\lin \perp} \otimes _ k R = \LL_d \otimes_k R. $$
 Thus we see that the set of $R$-points of $\mathcal V$ lifting $x_0$ is in canonical bijection with the set of $R$-points of $\mathcal I$ lifting $x_0$. We thus obtain a canonical $\tilde {\mathcal R}_{p}$-point of $\mathcal V$ lifting $x_ 0 \in \mathcal V (k)$, and a canonical $\tilde {\mathcal S}_p$-point of $\mathcal I$ lifting $x_0 \in \mathcal I (k)$. These two points induce maps $\tilde{\mathcal S}_p \to \tilde{\mathcal R}_p$ and $\tilde{\mathcal R}_p \to \tilde{\mathcal S}_p$ respectively. From the moduli interpretation of these two maps we see that they are $k$-algebra homomorphisms inverse to each other and equivariant with respect to the actions of $\bar g$ on both sides. Note that $\mathcal S_p$ (resp. $\mathcal R_p$) is the quotient of $\tilde {\mathcal S}_p$ (resp. $\tilde{\mathcal R}_p$) by the augmentation ideal for the $\bar g$-action. It follows that $\mathcal R_p \cong \mathcal S_p$.   
 \end{proof}
\subsection{Study of $\mathcal I^{\bar g}$}\label{sec:study-mathcal-ibar}

\begin{defn} \label{defn:lambda}Let $\lambda$ be the eigenvalue of $\bar g$ on the $1$-dimensional $k$-vector space $\LL_{d+1}/\LL_d = U_0/ U_0^{\perp}$, and let $c$ be the size of the unique (cf. Remark \ref{rem:unique block}) Jordan block of $\bar g|_{\LL_{d+1}}$ associated to $\lambda$.  Notice our $c$ is denoted by $c+1$ in \cite[\S 9]{RTZ}.
\end{defn}

\begin{rem}\label{rem:lambda and c} By the discussion before \cite[Proposition 9.1]{RTZ}, $c$ is the size of the unique Jordan block associated to $\lambda$ of $\bar g$ on ${\Theta/ \LL_d} = \Omega/U_0^{\perp}$, and is also equal to the quantity $\frac{m(Q(T))+1}{2}$ introduced in the introduction.
\end{rem}
\begin{prop}\label{prop:compute R}
The local ring $\oo_{\mathcal I^{\bar g} ,{x_0}}$ of $\mathcal I^{\bar g}$ at $x_0$ is isomorphic to $k[X]/X^{c}$ as a $k$-algebra.
\end{prop}
\begin{proof}
By Remark \ref{rem:lambda and c} and Definition \ref{eq:defn of I}, the proposition is a consequence of the following general lemma applied to $\LL = \Theta/\LL_d$ and $h = \bar g$.
\end{proof}

\begin{lem}\label{lem:jordanblock}
  Let $\LL$ be a $k$-vector space of dimension $d+1$. Let $\mathbb{P}(\LL)=\mathbb{P}^d$ be the associated projective space. Let $x_0 \in \mathbb{P}(\mathcal{L})(k)$, represented by a vector $\ell\in \LL$. Let $h \in \GL(\LL)(k)=\GL_{d+1}(k)$. Assume that
  \begin{enumerate}
  \item the natural action of $h$ on $\mathbb{P}(\mathcal{L})$ fixes $x_0$. Denote the eigenvalue of $h$ on $\ell$ by $\lambda$.
  \item there is a unique Jordan block of $h$ associated to the eigenvalue $\lambda$. Denote its size by $c$.
  \end{enumerate}
   Let $R:= \oo_{\mathbb{P}(\mathcal{L})^h, x_0}$ be the local ring of the fixed point scheme $\mathbb{P}(\mathcal{L})^h$ at $x_0$. Then $$R\cong k[X]/X^{c}.$$
\end{lem}

\begin{proof}
  Extend $\ell$ to a basis $\{\ell_0=\ell, \ell_1,\ldots, \ell_d\}$ of $\mathcal{L}$ such that the matrix $(h_{ij})_{0\le i,j\le d}$ of $h$ under this basis is in the Jordan normal form. Under this basis, the point $x_0$ has projective coordinates $[X_0:\cdots:X_d]=[1:0:\cdots:0]\in \mathbb{P}^d$. Let $Z_i=X_i/X_0$ ($1\le i\le d$) and let $\mathbb{A}^d$ be the affine space with  coordinates $(Z_1,\ldots, Z_d)$. Then we can identify the local ring of $\mathbb{P}^d$ at $x_0$ with the local ring of $\mathbb{A}^d$ at the origin. Since $h$ fixes $x_0$, we know that $h$ acts on the local ring of $\mathbb{A}^d$ at the origin (although $h$ does not stabilize $\adele^d$ in general). Since $(h_{ij})$ is in the Jordan normal form, we know that the action of $h$ on the latter is given explicitly by $$h Z_i=\frac{h_{i,i}X_i+h_{i,i+1}X_{i+1}}{h_{0,0}X_0+ h_{0,1}X_1}=\frac{h_{i,i}Z_i+ h_{i, i+1} Z_{i+1}}{h_{0,0} + h_{0,1}Z_1},\ 1\le i\le d,$$ where $h_{i,i+1}Z_{i+1}$ is understood as 0 when $i=d$. Hence the local equations at the origin of $\adele^d$ which cut out the $h$-fixed point scheme are given by $$(h_{0,0}-h_{i,i}) Z_i+ h_{0,1}Z_1 Z_i=h_{i,i+1}Z_{i+1},\ 1\leq i \leq d.$$  By hypothesis (2), we have $h_{0,0}-h_{i,i}\ne 0$ if and only if $i\ge c$. Thus when $i\ge c$, we know that $(h_{0,0}-h_{i,i})+h_{0,1}Z_1$ is a unit in the local ring of $\mathbb{A}^d$ at the origin, and so $Z_{i}$ can be solved as a multiple of $h_{i,i+1}Z_{i+1}$ when $i\ge c$. It follows that $$\quad Z_{i}=0,\ i\ge c.$$ If $c=1$, then $Z_1=\cdots =Z_d=0$ and the local ring $R$ in question is isomorphic to $k$ as desired. If $c>1$, then $h_{0,1}=1$ and we find the equations for $i=1,\cdots, c-1$ simplify to $$Z_1Z_1=Z_2,Z_1 Z_2=Z_3,\cdots, Z_1Z_{c-2}=Z_{c-1},Z_1 Z_{c-1}=0.$$ Hence the local ring $R$ in question is isomorphic to (the localization at the ideal $(Z_1, Z_2,\cdots, Z_{c-1})$ or $(Z_1)$ of) $$k[Z_1,Z_2,\ldots, Z_{c-1}]/(Z_1^2-Z_2, Z_1^3-Z_3, \cdots, Z_1^{c-1}-Z_{c-1}, Z_1^c)\cong k[Z_1]/Z_1^c,$$ as desired. 
\end{proof}

\begin{thm}\label{thm:finalA}
	Let $g \in J(\QQ_p)$ be regular semisimple and minuscule. Let $x_0 \in (\delta(\mathcal M)\cap \RZ^g)(k)$. Also denote by $x_0$ the image of $x_0$ in $\mathcal V(\Lambda) (k)$ as in Proposition \ref{prop:scheme} and define $\lambda, c$ as in Definition \ref{defn:lambda}. Assume $p>c$. Then the complete local ring of $\delta(\mathcal M)\cap \RZ^g$ at $x_0$ is isomorphic to $k[X]/ X^{c}$. 
\end{thm}
\begin{proof}
Let $\hat {\mathcal S}$ be the complete local ring of $\delta(\mathcal M) \cap \mathcal N^g$ at $x_0$. By Proposition \ref{prop:scheme} and by the fact that $\mathcal V(\Lambda)$ is smooth of dimension $d$ (\S \ref{subsec:V(Lambda)}), we know that $\hat{\mathcal{ S}}$ is a quotient of the power series ring $k[[X_1,\cdots, X_d]]$. By Proposition \ref{prop:scheme}, Lemma \ref{lem:local model}, Proposition \ref{prop:compute R}, we know that $\hat {\mathcal S}/\mathfrak m _{ \hat {\mathcal S}} ^p$ is isomorphic to $k[X]/X^{c}$ as a $k$-algebra. In such a situation, it follows from the next abstract lemma that $\hat{\mathcal S} \cong k[X]/X^{c}$.
\end{proof}

\begin{lem}Let $I$ be a proper ideal of $k[[X_1,\cdots, X_d]]$ and let $\hat {\mathcal S} = k[[X_1,\cdots, X_{d}]]/I .$ Let $\mathfrak m$ be the maximal ideal of $k [[X_1,\cdots, X_{d}]]$ and let $ {\mathfrak  m }_{\hat{\mathcal S}}$ be the maximal ideal of $\hat {\mathcal S}$. Assume there is a $k$-algebra isomorphism $\beta:\hat{\mathcal S}/ {\mathfrak  m }_{\hat{\mathcal S}} ^p \isom k[X]/X^c$ for some integer $1 \leq c <p$. Then
	$\hat {\mathcal S}$ is isomorphic to $k[X]/ X^c$ as a $k$-algebra.
\end{lem}
\begin{proof}
		We first notice that if $R_1$ is any quotient ring of $k[[X_1,\cdots, X_{d}]]$ with its maximal ideal $\mathfrak m_1$ satisfying $\mathfrak m_1 = \mathfrak m_1^2$ (i.e. $R_1 $ has zero cotangent space), then $R_1 =k$. In fact, $R_1$ is noetherian and we have $\mathfrak m_1^l = \mathfrak m_1 $ for all $l\in \ZZ_{\geq 1}$, so by Krull's intersection theorem we conclude that $\mathfrak m_1 = 0$ and $R_1 =k$. 
		
		Suppose $c=1$. Then $\hat {\mathcal S}/ {\mathfrak  m }_{\hat{\mathcal S}} ^p \cong k$, so $\hat {\mathcal{S}}$ has zero cotangent space and thus $\hat {\mathcal S }= k$ as desired. Next we treat the case $c\geq 2$. Let $\alpha$ be the composite 
		$$\alpha: k [[ X_1,\cdots, X_{d}]] \to\hat { \mathcal S}/{\mathfrak  m }_{\hat{\mathcal S}}^p  \xrightarrow{\beta}  k[X]/X^c. $$ Let $J = \ker \alpha$. Since $\alpha$ is surjective, we reduce to prove that $I= J$. Note that because $\beta$ is an isomorphism we have \begin{align}\label{first reln between I and J}
		I +  {\mathfrak m }^p = J.
		\end{align}
		In the following we prove $\mathfrak m^p \subset I$, which will imply $I= J$ and hence the lemma. The argument is a variant of \cite[Lemma 11.1]{RTZ}. 
		
		Let $Y \in k [[ X_1,\cdots, X_{d}]]$ be such that $\alpha(Y) =X$. Since $X$ generates the maximal ideal in $k[X]/X^c$, we have 
		\begin{align}\label{m=J+Y} \mathfrak m =  J +(Y).
		\end{align} Then by (\ref{first reln between I and J}) and (\ref{m=J+Y}) we have $\mathfrak m = I + (Y) + \mathfrak m ^p$,
		and so the local ring $k[[X_1,\cdots, X_{d}]]/ (I+(Y))$ has zero cotangent space. We have observed that the cotangent space being zero implies that the ring has to be $k$, or equivalently
		\begin{align}\label{m=I+Y}
		\mathfrak m = I + (Y)  .
		\end{align}
		
		Now we start to show $\mathfrak m^p \subset I$. By (\ref{m=I+Y}) we have $\mathfrak m^p \subset I+ (Y^p)$, so we only need to prove $Y^p\in  I$. We will show the stronger statement that $Y^c \in I$. By Krull's intersection theorem, it suffices to show that $Y^c \in I + \mathfrak m ^{pl}$ for all $l\geq 1$. In the following we show this by induction on $l$.
		
		Assume $l=1$. Note that $\alpha (Y^c) =0$, so by (\ref{first reln between I and J} ) we have $$ Y^c \in J=  I + \mathfrak m^p.$$ Suppose $Y^ c \in  I + \mathfrak m^{pl}$ for an integer $ l \geq 1$. Write
		\begin{align}\label{decomp Y^c}
		Y^c = i +  m, ~i \in I,~ m \in \mathfrak m^{pl}.
		\end{align}   By (\ref{m=J+Y}) we know $$\mathfrak m^{pl} \subset ( J + (Y)) ^{pl} \subset \sum _{s=0}^{pl} J^s (Y) ^{pl-s}. $$ Thus we can decompose $m\in \mathfrak m ^{pl}$ into a sum 
		\begin{align}\label{decomp m}
		m = \sum_{s=0} ^{pl} j_s Y^{pl-s},~ j_s \in J^s.
		\end{align}
		By (\ref{decomp Y^c}) and (\ref{decomp m}), we have 
		\begin{align*}
		Y^c = i +  \sum _{s=0}^{pl} j_s Y ^{pl-s} ,
		\end{align*}
and so 
\begin{align}\label{to extract A}
		Y^c - \sum _{s=0}^{pl-c} j_s Y ^{pl-s} =  i + \sum _{s=pl-c+1}^{pl} j_s Y ^{pl-s}.
		\end{align}
		
		Denote $$A : = \sum_{s =0 } ^{pl-c} j_s Y^{pl-s-c} .$$ Then the left hand side of (\ref{to extract A}) is equal to $(1-A) Y^c$. Hence we have $$(1-A)  Y^c = i + \sum_{ s= pl-c+1} ^{pl} j_s Y^{pl -s}  \subset I + J^{pl-c+1} \xlongequal{(\ref{first reln between I and J})} I+ (I+\mathfrak m^p) ^{pl-c+1} = I + \mathfrak m ^{p (pl-c+1)}  \subset I + \mathfrak m ^{p(l+1)}, $$ where for the last inclusion we have used $c< p$. Since $1-A$ is a unit in $k[[X_1,\cdots, X_{d}]] $ (because $c<p$), we have $Y^c \in I + \mathfrak m^{p(l+1)}$. By induction, $Y^c \in I + \mathfrak m^{pl}$ for all $l\in \ZZ_{\geq 1}$, as desired.
	\end{proof}

\bibliographystyle{hep}
\bibliography{myref}
	
\end{document}